\documentclass[preprint,10pt]{elsarticle}
\usepackage{amsthm,amsmath,amssymb,mathrsfs}
\usepackage[colorlinks=true,citecolor=black,linkcolor=black,urlcolor=blue]{hyperref}
\usepackage{graphicx}
\usepackage{float}
\usepackage{epstopdf}
\usepackage{epsfig}
\usepackage{extarrows,chngpage,array,float,eqnarray}     

\journal{Journal of Combinatorial Theory, Series B}
\theoremstyle{plain}
\newtheorem{theorem}{Theorem}[section]
\newtheorem{lemma}[theorem]{Lemma}
\newtheorem{corollary}[theorem]{Corollary}

\theoremstyle{definition}
\newtheorem{definition}[theorem]{Definition}
\newtheorem{example}[theorem]{Example}

\newtheorem{problem}[theorem]{Problem}
\newtheorem{proposition}[theorem]{Proposition}

\theoremstyle{remark}
\newtheorem{remark}[theorem]{Remark}

\allowdisplaybreaks

\title{On the coefficients of interior and exterior polynomials of polymatroids}
\author{Xiaxia Guan$^{a}$,~~Xian'an Jin$^{b}$\footnote{Corresponding author.},~~Tianlong Ma$^{c}$,~~Weihua Yang$^{a}$ \\
\small $^a$Department of Mathematics, Taiyuan University of Technology, P. R. China\\
\small $^b$School of Mathematical Sciences, Xiamen University, P. R. China\\
\small $^c$School of Science, Jimei University, P. R. China\\
\small \emph{Email addresses}: guanxiaxia@tyut.edu.cn; xajin@xmu.edu.cn; tianlongma@aliyun.com; yangweihua@tyut.edu.cn}

\begin{document}
\begin{abstract}
The Tutte polynomial is an important invariant of  graphs and matroids. Chen and Guo \emph{[Adv. in Appl. Math. 166 (2025) 102868.]} proved that
for a $(k+1)$-edge connected graph $G$ and  for any $i$ with $0\leq i <\frac{3(k+1)}{2}$, $$[y^{g-i}]T_{G}(1,y)=\binom{|V(G)|+i-2}{i}-\sum_{j=0}^{i}\binom{|V(G)|+i-2-j}{i-j}|\mathcal{SC}_{j}(G)|,$$
where $g=|E(G)|-|V(G)|+1$, $\mathcal{SC}_{j}(G)$ is the set of all minimal edge cuts with $j$ edges, $T_{G}(x,y)$ is the Tutte polynomial of the graph $G$, and $[y^{g-i}]T_{G}(1,y)$ denotes the coefficient of $y^{g-i}$ in the polynomial $T_{G}(1,y)$. Recently, Ma, Guan and Jin \emph{[arXiv.2503.06095, 2025.]} generalized this result from graphs to matroids and obtained the dual result on coefficients of $T_M(x,1)$ of matroids $M$ at the same time. In 2013, as a generalization of $T_{G}(x,1)$ and $T_{G}(1,y)$ of graphs $G$ to hypergraphs,  K\'{a}lm\'{a}n \emph{[Adv. Math. 244 (2013) 823-873.]} introduced   interior  and  exterior polynomials for connected hypergraphs. Chen and Guo posed a problem that can one generalize these results of graphs to interior  and  exterior polynomials of hypergraphs? In this paper, we solve it in the affirmative by obtaining  results for more general polymatroids, which include the case of hypergraphs and also generalize the results of matroids due to Ma, Guan and Jin. As an application, the sequence consisting of these coefficients on polymatroids is proven to be unimodal, while the unimodality of the whole coefficients of matroids was obtained in 2018 by Adiprasito, Huh and Katz using Hodge theory.
\end{abstract}


\maketitle
\section{Introduction}
\noindent

The Tutte polynomial \cite{Tutte} is a crucial and well-studied topic in graph and matroid theory (Crapo \cite{Crapo} extended the Tutte polynomial from graphs to matroids.), having wide applications in combinatorics, statistical physics, knot theory, quantum group theory and so on. Let $T_{G}(x,y)$ be the Tutte polynomial of a graph $G$ and let $[x^{i}]f(x)$ denote the coefficient of $x^{i}$ in a polynomial $f(x)$. Recently, Chen and Guo  \cite{Chen} obtained a relation between  coefficients of $T_{G}(1,y)$ and its minimal edge cuts of the graph $G$ via the Abelian sandpile model.

\begin{theorem} \label{connected-}\cite{Chen}
Let $G$ be a $(k+1)$-edge connected graph and let $g=|E(G)|-|V(G)|+1$. Then for any $i$ with $0\leq i <\frac{3(k+1)}{2}$, $$[y^{g-i}]T_{G}(1,y)=\binom{|V(G)|+i-2}{i}-\sum_{j=0}^{i}\binom{|V(G)|+i-2-j}{i-j}|\mathcal{SC}_{j}(G)|,$$
where $\mathcal{SC}_{j}(G)$ is the set of all minimal edge cuts with $j$ edges.
\end{theorem}

In fact, $\mathcal{SC}_{j}(G)=\emptyset$ for any $0\leq j \leq k$ in case that $G$ is a $(k+1)$-edge connected graph.
Since for a plane graph $G$, $T_{G}(x,y)=T_{G^{*}}(y,x)$, and  minimal edge cuts of $G$  and cycles of $G^{*}$ are one-to-one correspondence, where $G^{*}$ is the dual graph of $G$, they \cite{Chen} first posed the following dual problem.

\begin{problem}\label{circuits}\cite{Chen}
Can one express the coefficients of $T_{G}(x,1)$ using cycles?
\end{problem}

It is well-known that any graph can induce a graphical matroid, i.e. cycle matroid. In this sense matroids generalize graphs. They \cite{Chen} also posed a problem as follows.
\begin{problem}\label{matroids}\cite{Chen}
Can one generalize the results of Theorem \ref{connected-} and Problem \ref{circuits} to matroids?
\end{problem}

Let $T_{M}(x,y)$ be the Tutte polynomial of a matroid $M$. Problems \ref{circuits} and \ref{matroids} have been solved in \cite{Ma} by establishing the relation between $T_{M}(1,y)$ and hyperplanes of matroids $M$ and duality relation of the Tutte polynomial. We state the result.

\begin{theorem} \label{connected}\cite{Ma}
Let $M$ be a matroid of rank $d$ over $E$. Then
 \begin{itemize}
  \item [(i)] for any $i$ with $0\leq i <f_{2}(M)$, $$[y^{|E|-d-i}]T_{M}(1,y)=\binom{d+i-1}{i}-\sum_{j=0}^{i}\binom{d+i-1-j}{i-j}|\mathcal{H}_{j}(M)|;$$
  \item [(ii)] for any $i$ with $0\leq i <f'_{2}(M)$, $$[x^{d-i}]T_{M}(x,1)=\binom{|E|-d+i-1}{i}-\sum_{j=0}^{i}\binom{|E|-d+i-1-j}{i-j}|\mathcal{C}_{j}(M)|,$$
\end{itemize}
where $\mathcal{H}_{j}(M)$ is the set of all hyperplanes with $|E|-j$ elements,  $\mathcal{C}_{j}(M)$ is the set of all circuits with $j$ elements, $f_{k}(M)=\min\{|F|:~\text{the rank of}~E\setminus F~\text{is}~d-k\}$, and $f'_{k}(M)=\min\{|F|:~\text{the rank of}~ F~\text{is}~|F|-k\}$.
\end{theorem}

In fact, $\mathcal{H}_{j}(M)=\emptyset$ if $j<f_{1}(M)$, and $\mathcal{C}_{j}(M)=\emptyset$ if $j<f'_{1}(M)$. They \cite{Ma} also showed that $\frac{3(k+1)}{2}\leq f_{2}(M(G))$ for any $(k+1)$-edge connected graph $G$, where $M(G)$ is the cycle matroid of $G$. Hence, Theorem \ref{connected}(i)  generalizes the result of Theorem \ref{connected-} from graphs to matroids. As a generalization of the one-variable evaluations $T_G(x,1)$ and $T_G(1,y)$ of the Tutte polynomial $T_G(x,y)$ of graphs $G$ to hypergraphs, K\'{a}lm\'{a}n \cite{Kalman1} introduced the interior polynomial $I_{\mathcal{H}}(x)$ and the exterior polynomial $X_{\mathcal{H}}(y)$ for connected hypergraphs $\mathcal{H}$ via internal and external activities of hypertrees. (Hypertrees were first described as `left or right degree vectors' in \cite{Postnikov}.)  Later, K\'{a}lm\'{a}n, Murakami, and Postnikov \cite{Kalman2,Kalman3} established  that certain leading terms of the HOMFLY polynomial \cite{Jones}, which is a generalization of the celebrated Jones polynomial \cite{Jones} in knot theory, of any special alternating link coincide with the common interior polynomial of the pair of hypergraphs derived from the Seifert graph (which is a bipartite graph) of the link. Chen and Guo \cite{Chen} also posed the following problem for interior and exterior polynomials of hypergraphs.

\begin{problem}\label{exterior}\cite{Chen}
Can one generalize the results of Theorem \ref{connected-} and Problem \ref{circuits} to interior and exterior polynomials of hypergraphs?
\end{problem}

Polymatroids are a generalization of matroids and an abstraction of hypergraphs. In 2022, Bernardi, K\'{a}lm\'{a}n and Postnikov \cite{Bernardi} extended the interior polynomial $I_{P}(x)$ and the exterior polynomial $X_{P}(y)$ from hypergraphs to polymatroids $P$, which also generalize the one-variable evaluations $T_M(x,1)$ and $T_M(1,y)$ of the Tutte polynomial $T_M(x,y)$ of matroids $M$, respectively. In \cite{Guan4}, coefficients of some lower-order terms of the exterior polynomial were obtained for polymatroids.

\begin{corollary} \label{connected-exterior} \cite{Guan4}
Let $P\subseteq \mathbb{Z}^{n}_{\geq 0}$ be a polymatroid and $f$ be its rank function. Then for all $k\leq n-1$, the coefficient $$[y^{i}]X_{P}(y)=\binom{f([n])+i-1}{i}$$ for all $i\leq k$ if and only if $f([n]\setminus J)=f([n])$ for all $J\subseteq [n]$ with $|J|=k$.
\end{corollary}

 In this paper, we answer Problem \ref{exterior} in the affirmative by proving a general result for polymatroids in Theorem \ref{connected-exterior2}, which also extends Theorem \ref{connected} from matroids to polymatroids, and includes Corollary \ref{connected-exterior} as a special case.

In 2018, Adiprasito, Huh and Katz \cite{Adiprasito} indirectly proved that the sequences of all coefficients of $T_{M}(x,1)$ and $T_{M}(1,y)$ are both unimodal by applying Hodge theory. In this paper,  the unimodality of partial coefficients obtained in Theorem \ref{connected-exterior2}, of  exterior and interior polynomials for polymatroids, is proven. This not only implies the unimodality of partial coefficients obtained in Theorems \ref{connected-} and \ref{connected}, but also provides a simpler and direct proof.

The paper is organized as follows. In Section 2, we provide some necessary definitions. Section 3 is devoted to solve Problem \ref{exterior}. The unimodality of these coefficients is proven in Section 4.
\section{Preliminaries}
\noindent

In this section, we will give some definitions.
Throughout the paper, let  $[n]=\{1,2,\ldots,n\}$, $2^{[n]}=\{I|I\subseteq[n]\}$, and let $\textbf{e}_{1},\textbf{e}_{2},\ldots,\textbf{e}_{n}$ denote the canonical basis of $\mathbb{R}^{n}$.
We first recall the definition of polymatroids.
\begin{definition} \label{def polymatroid}
A \emph{polymatroid} $P=P_{f}\subseteq \mathbb{Z}^{n}_{\geq 0}$ (in other words, over the ground set $[n]$) with rank function $f$ is  given as $$P=\left\{(a_{1},\ldots,a_{n})\in \mathbb{Z}^{n}_{\geq 0}\bigg|\sum_{i\in I}a_{i}\leq f(I) \ \text{for any}\ I\subseteq [n] \ \text{and} \ \sum_{i\in [n]}a_{i}=f([n])\right\},$$ where   $f:2^{[n]}\rightarrow \mathbb{Z}_{\geq 0}$ satisfies
\begin{enumerate}
\item[(i)] $f(\emptyset)=0$ (normalization);
\item[(ii)] $f(I)\leq f(J)$ for any $I\subseteq J\subseteq [n]$ (monotonicity);
\item[(iii)] $f(I)+f(J)\geq f(I\cup J)+f(I\cap J)$ for any $I,J\subseteq [n]$ (submodularity).
\end{enumerate}
We say that a vector $\textbf{a}$ is a \emph{basis} of the polymatroid $P$ if $\textbf{a}\in P$.
\end{definition}
\vskip0.1cm
It is easy to see that (the set of its vertices of) the base polytope of any matroid
is a polymatroid. More precisely, if $M$ is a matroid over $[n]$, and  $P=P(M)$ is the corresponding polymatroid of $M$, then for any $B\subseteq [n]$, let
$$I_{B}=(a_{1},\ldots,a_{n})\in \{0,1\}^{n}\ \text{where}\ a_{i}=1 \ \text{if and only if}\  i\in B.$$ Then
$$P=P(M)=\{I_{B} ~|~\text{$B$ is a base of} \ M\}\subseteq \{0,1\}^{n}\subseteq \mathbb{Z}^{n}_{\geq 0}.$$

\begin{definition}
A \emph{hypergraph} is a  pair $\mathcal{H}=(V,E)$, where $V$ is a finite set and $E$ is a finite multiset of non-empty subsets of $V$. Elements of $V$ are called \emph{vertices} and  elements of $E$ are called \emph{hyperedges}, respectively, of the hypergraph.
\vskip0.1cm
For a hypergraph $\mathcal{H}=(V,E)$, its \emph{associated bipartite graph} $\mathrm{Bip} \mathcal{H}$ is defined as follows.
The sets $V$ and $E$ are the partite sets of $\mathrm{Bip} \mathcal{H}$, and an element $v$ of $V$ is adjacent to an element $e$ of $E$ in $\mathrm{Bip} \mathcal{H}$ if and only if $v\in e$.

For a subset $E'\subseteq E$, let $\mathrm{Bip} \mathcal{H}|_{E'}$ denote the bipartite graph formed by $E'$, all edges of $\mathrm{Bip} \mathcal{H}$ incident with elements of $E'$, and their end-vertices in $V$. Define $\mu(E')=0$ for $E'=\emptyset$, and $\mu(E')=|\bigcup E'|-c(E')$ for $E'\neq \emptyset$, where $\bigcup E'=V\cap (\mathrm{Bip} \mathcal{H}|_{E'})$ and $c(E')$ is the number of connected components of $\mathrm{Bip} \mathcal{H}|_{E'}$.
\end{definition}

K\'{a}lm\'{a}n \cite{Kalman1} proved that $\mu$ is submodular and monotonous.
In that sense, polymatroids generalize hypergraphs. The elements of the polymatroid induced by $\mu$ (in the sense of Definition \ref{def polymatroid}) will be referred  to as  \emph{hypertrees} because they are essentially degree distributions on $E$ of maximal spanning forests of $\mathrm{Bip} \mathcal{H}$, cf. \cite{Kalman1} and also generalizations of spanning trees of a connected graph.

\begin{definition}
A polymatroid is called a \emph{hypergraphical polymatroid}, denoted by $P_{\mathcal{H}}$, if it is the set of all hypertrees of some hypergraph $\mathcal{H}$.
\end{definition}

We next recall (internal and  external) activity of a basis of a polymatroid.
\begin{definition}
Let $P$ be a polymatroid  over $[n]$.
For a basis $\textbf{a}\in P$, an index $i\in [n]$ is \emph{internally active} if $\textbf{a}-\textbf{e}_{i}+\textbf{e}_{j}\notin P$ for any $j<i$. Let $\mathrm{Int}(\textbf{a})=\mathrm{Int}_{P}(\textbf{a})\subseteq [n]$ denote the set of all
internally active indices with respect to $\textbf{a}$.

We call $i\in [n]$ \emph{externally active} if $\textbf{a}+\textbf{e}_{i}-\textbf{e}_{j}\notin P$ for any $j<i$. Let $\mathrm{Ext}(\textbf{a})=\mathrm{Ext}_{P}(\textbf{a})$ denote the set of all externally active indices with respect to $\textbf{a}$.

$|\mathrm{Int}(\textbf{a})|$ and $|\mathrm{Ext}(\textbf{a})|$ are called  \emph{internal and  external activity} of $\textbf{a}$, respectively.
\end{definition}
K\'{a}lm\'{a}n \cite{Kalman1} defined  interior and exterior polynomials of hypergraphical polymatroids. Recently, Bernardi, K\'{a}lm\'{a}n and Postnikov \cite{Bernardi} extended to polymatroids.
\begin{definition}\label{def}
For a polymatroid $P\subseteq \mathbb{Z}^{n}_{\geq 0}$, denote \emph{the interior polynomial} $$I_{P}(x):=\sum_{\textbf{a}\in P}x^{n-|\mathrm{Int}(\textbf{a})|}$$ and
 \emph{the exterior polynomial} $$X_{P}(y):=\sum_{\textbf{a}\in P}y^{n-|\mathrm{Ext}(\textbf{a})|}.$$
\end{definition}

\begin{remark} We have two notes on Definition \ref{def}.
\begin{itemize}
\item[(a)] In fact, Bernardi \emph{et al.}~\cite{Bernardi} defined  interior and exterior polynomials of polymatroids $P\subseteq\mathbb{Z}^{n}$. For any $\textbf{c}=(c_{1},\ldots,c_{n})\in \mathbb{Z}^{n}$, let the polymatroid
$P+\textbf{c} : =\{(a_{1}+c_{1},\ldots,a_{n}+c_{n})|(a_{1},\ldots,a_{n})\in P\}$.
Bernardi  \emph{et al.}~\cite{Bernardi} proved
\begin{eqnarray}\label{translation invariance}
I_{P}(x)=I_{P+\textbf{c}}(x)~\text{and}~X_{P}(y)=X_{P+\textbf{c}}(y).
\end{eqnarray}
Hence, it is enough to consider interior and exterior polynomials for polymatroids $P\subseteq \mathbb{Z}^{n}_{\geq 0}$.
\item[(b)] For a polymatroid $P\subseteq \mathbb{Z}^{n}_{\geq 0}$ over $[n]$ with rank function $f$, let $f^{\ast}$ be a function with
\begin{eqnarray} \label{dual-polymatroid}
f^{\ast}(I)=f([n]\setminus I)-f([n])+\sum_{i'\in I}f(\{i'\})
\end{eqnarray}
for any $I\subseteq [n]$. It is easy to see that $f^{\ast}$ satisfies (i)-(iii) of the rank funtion in Definition \ref{def polymatroid}. Let $P^{\ast}$ denote the polymatroid  given by $f^{\ast}$. Then $$P^{\ast}=-P+\textbf{c},$$ where $-P:=\{(-a_{1},\ldots,-a_{n})|(a_{1},\ldots,a_{n})\in P\}$ and $c_{i}=f(\{i\})$ for any $i\in [n]$. Bernardi  \emph{et al.}  \cite{Bernardi} also proved $I_{P}=X_{-P}$. Hence, by Equation (\ref{translation invariance}),
\begin{eqnarray}\label{duality-relation}
I_{P}=X_{P^{\ast}}~\text{and}~I_{P^{\ast}}=X_{P}.
\end{eqnarray}
\end{itemize}
\end{remark}
\vskip0.1cm
The definition of the interior polynomial $I_{P}$ and the exterior polynomial $X_{P}$ for a polymatroid $P$ relies on the order $1<2<\ldots<n$, but Bernardi  \emph{et al.}~\cite{Bernardi} showed that both $I_{P}$ and $X_{P}$ depend only on $P$.

Let $M$ be a matroid of rank $d$ over $[n]$, and let $P=P(M)$ be its corresponding polymatroid. Bernardi  \emph{et al.}~\cite{Bernardi} also proved that
\begin{equation*}\label{duality}
I_{P(M)}(x)=x^{d}T_{M}(x^{-1},1)~\text{and}~X_{P(M)}(y)=y^{n-d}T_{M}(1,y^{-1}).
\end{equation*}

\section{Main results}
\noindent

In this section, we study coefficients of interior and exterior polynomials of polymatroids to solve  Problem \ref{exterior}. Recall that $[x^{i}]f(x)$ is the coefficient of $x^{i}$ in the  polynomial $f(x)$. We start with some known results.

\begin{lemma} \label{coefficients-12}\cite{Guan3}
Let $P\subseteq \mathbb{Z}^{n}_{\geq 0}$ be a polymatroid with rank function $f$. Then
\begin{itemize}
  \item [(i)] $[y^{0}]X_{P}(y)=1$;
  \item [(ii)] $[y]X_{P}(y)=\sum_{i\in [n]}f([n]\setminus \{i\})-(n-1)f([n])$.
\end{itemize}
\end{lemma}

Let $P$ be a polymatroid over $[n]$ and let $f$ be its rank function. For convenience, for any $t\in [n]$, let $\alpha_{t}=f([n])-f([n]\setminus \{t\})$, $\beta_{t}=f(\{t\})$ and $T_{t}=\{\alpha_{t}, \alpha_{t}+1,\ldots,\beta_{t}\}$. For any $j\in T_{t}$, define
$$P^{t}_{j}:=\{(a_{1},\ldots,a_{n})\in P\mid a_{t}=j\}$$ and its projection
\begin{equation*}
\widehat{P}^{t}_{j}:=\{(a_{1},\ldots,a_{t-1},a_{t+1},\ldots,a_{n})\in \mathbb{Z}^{n-1}\mid (a_{1},\ldots,a_{n})\in P^{t}_{j}\}.
\end{equation*}
 Here the range $T_{t}$ is chosen so that $P^{t}_{j}$ and $\widehat{P}^{t}_{j}$ are nonempty if and only if $j\in T_{t}$. In \cite{Guan4}, the authors showed that $P^{t}_{j}$ and $\widehat{P}^{t}_{j}$ are polymatroids on $[n]$ and on $[n]\setminus\{t\}$, respectively. In fact, they \cite{Guan4} obtained the relation between the rank function of $P$ and the rank function of $\widehat{P}^{t}_{j}$ as follows.

\begin{proposition}\label{rank function of Pj}\cite{Guan4}
Let $P\subseteq \mathbb{Z}^{n}_{\geq 0}$ be a polymatroid over $[n]$ with rank function $f$. For some $t\in [n]$ and for any $j\in T_{t}$, let $f^{t}_{j}$ be the rank function of the polymatroid $\widehat{P}^{t}_{j}$. Then for any subset $I\subseteq [n]\setminus \{t\}$, we have $$f^{t}_{j}(I)=\min\{f(I), f(I\cup \{t\})-j\}.$$
\end{proposition}

\begin{definition}
Let $P$ be a polymatroid on $[n]$ with rank function $f$. For a subset $A\subseteq [n]$, the \emph{deletion} $P\setminus A$ and \emph{contraction} $P/A$, which are polymatroids on $[n]\setminus A$, are given by the rank functions $f_{P\setminus A}(I)=f(I)$ and $f_{P/A}(I)=f(I\cup A)-f(A)$, for any subset $I\subseteq [n]\setminus A$, respectively. We replace $P\setminus \{t\}$ and $P/\{t\}$ with $P\setminus t$ and $P/t$, respectively, if $A=\{t\}$.
 \end{definition}
By the submodularity of $f$, for any subset $I\subseteq [n]\setminus \{t\}$,  $$f(I\cup \{t\})+f([n]\setminus \{t\})\geq f([n])+f(I),$$ $$f(I)+f(\{t\})\geq f(I\cup \{t\}).$$
Hence, Proposition \ref{rank function of Pj} implies that
 \begin{eqnarray*}\label{Del-Con}
 \widehat{P}^{t}_{\alpha_{t}}=P\setminus t\ \text{and} \ \widehat{P}^{t}_{\beta_{t}}=P/t.
\end{eqnarray*}

In \cite{Guan4}, a reduction formula of the exterior polynomial similar to famous deletion-contraction formula was obtained.
\begin{lemma} \label{deletion-contraction-IX}\cite{Guan4}
Let $P\subseteq \mathbb{Z}^{n}_{\geq 0}$ be a polymatroid over $[n]$. For any $t\in [n]$, $$X_{P}(y)=X_{P/t}(y)+y\sum_{j\in T_{t}\setminus \{\beta_{t}\}}X_{\widehat{P}^{t}_{j}}(y).$$
\end{lemma}

\begin{remark}
In fact, the results in Proposition \ref{rank function of Pj}, Lemmas \ref{coefficients-12} and \ref{deletion-contraction-IX} also hold for any polymatroid $P\subseteq\mathbb{Z}^{n}$.
\end{remark}

For a non-negative integer $k$, denote
$$r_{k}(P)=\min\{n-|F|:~f(F)\leq f([n])-k ~\text{for}~F\subseteq [n]\},$$ and
$$r'_{k}(P)=\min\left\{|F|:~\sum _{i'\in F}f(\{i'\})-f(F)\geq k ~\text{for}~F\subseteq [n]\right\}.$$

Note that $\sum _{i'\in \emptyset}f(\{i'\})=f(\emptyset)$. Then the following properties hold by the monotonicity of $f$, and the definitions of $r_{k}(P)$ and $r'_{k}(P)$.
\begin{proposition}\label{rP+r'P}
Let $P\subseteq \mathbb{Z}^{n}_{\geq 0}$ be a polymatroid over $[n]$ with  rank function $f$. For a non-negative integer $k$, we have that
\begin{enumerate}
\item [(i)] $r_{k}(P)$ exists if and only if $k\leq f([n])$, and $r'_{k}(P)$ exists if and only if $k\leq \sum _{i'\in [n]}f(\{i'\})-f([n])$;
\item [(ii)]  $0=r_{0}(P)\leq r_{k}(P)\leq r_{k+1}(P)\leq n$ and $0=r'_{0}(P)\leq r'_{k}(P)\leq r'_{k+1}(P)\leq n$;
\item [(iii)] $f([n]\setminus I)\geq f([n])-k+1$ for any $I$ with $|I|<r_{k}(P)$, and $\sum _{i'\in I}f(\{i'\})-f(I)\leq k-1$ for any $I$ with $|I|<r'_{k}(P)$.
\end{enumerate}
\end{proposition}

\begin{remark}
It is possible that for some $k$, there does not exist an $F\subseteq [n]$ such that $f(F)=f([n])-k$ or $\sum _{i'\in F}f(\{i'\})-f(F)=k$. Moreover, it is possible that $r_{k+1}(P)= r_{k}(P)$ or $r'_{k}(P)=r'_{k+1}(P)$ for some $k$.
\end{remark}

\begin{definition}
Let $P$ be a polymatroid  over $[n]$  and let $f$ be its rank function. For any $I\subseteq [n]$, its \emph{closure} is $$cl(I)=\{e\in [n]~|~f(I\cup \{e\})=f(I)\}.$$
\end{definition}
Clearly, $I\subseteq cl(I)$. We say that a subset $I\subseteq [n]$ is a \emph{flat} if $cl(I)=I$, that is, $f(I\cup \{e\})>f(I)$ for any $e\in [n]\setminus I$ by the monotonicity of $f$.

Let $$\mathcal{C}(P)=\left\{C\subseteq [n]:\sum _{i'\in C}f(\{i'\})-f(C)=1,\sum _{i'\in C\setminus j}f(\{i'\})=f(C\setminus j)~\text{for any}~j\in  C\right\}.$$
For a positive integer $j$, denote
$$\mathcal{H}_{j}(P)=\left\{H:~H ~\text{is a flat, $f(H)=f([n])-1$ and}~n-|H|=j\right\},$$
and
$$\mathcal{C}_{j}(P)=\left\{C: C\in \mathcal{C}~\text{and}~|C|=j\right\}.$$

Recall that $P^{\ast}$ is a polymatroid given by  $f^{\ast}$ (see Equation (\ref{dual-polymatroid})). By $f^{\ast}(I)=f([n]\setminus I)-f([n])+\sum_{i'\in I}f(\{i'\})$ for any $I\subseteq [n]$, we have that $f([n]\setminus I)=f([n])-k$ if and only if  $f^{\ast}(I)=\sum_{i'\in I}f(\{i'\})-k$. Then we have the following dual result.
\begin{proposition}\label{P+P*}
Let $P\subseteq \mathbb{Z}^{n}_{\geq 0}$ be a polymatroid over $[n]$ with rank function $f$.   Then
\begin{enumerate}
\item [(i)] $r'_{k}(P)=r_{k}(P^{\ast})$ and $r_{k}(P)=r'_{k}(P^{\ast})$  for a non-negative integer $k$;
\item [(ii)] $\mathcal{C}_{j}(P)=\{[n]\setminus H|~H\in \mathcal{H}_{j}(P^{\ast})\}$.
\end{enumerate}
\end{proposition}

We now can state our main result and prove it by applying Lemmas \ref{coefficients-12} and \ref{deletion-contraction-IX}, and Propositions \ref{rank function of Pj}, \ref{rP+r'P} and \ref{P+P*}.

\begin{theorem} \label{connected-exterior2}
Let $P\subseteq \mathbb{Z}^{n}_{\geq 0}$ be a polymatroid with rank function $f$ and $g=\sum _{i'\in [n]}f(\{i'\})-f([n])$. Then
 \begin{itemize}
  \item [(i)] for all $0\leq i<r_{2}(P)$, $$[y^{i}]X_{P}(y)=\binom{f([n])+i-1}{i}-\sum_{j=0}^{i}\binom{f([n])+i-1-j}{i-j}|\mathcal{H}_{j}(P)|;$$
  \item [(ii)] for all $0\leq i<r'_{2}(P)$, $$[x^{i}]I_{P}(x)=\binom{g+i-1}{i}-\sum_{j=0}^{i}\binom{g+i-1-j}{i-j}|\mathcal{C}_{j}(P)|.$$
\end{itemize}
\end{theorem}

\begin{proof}
We firstly prove the conclusion (i) by induction on $n$. Note that $\mathcal{H}_{0}(P)=\emptyset$. By Lemma \ref{coefficients-12}(i), we have that $[y^{0}]X_{P}(y)=1=\binom{f([n])+0-1}{0}$.

If $1<r_{2}(P)$, then by Proposition \ref{rP+r'P}(iii), for any $i'\in [n]$, we have that $f([n]\setminus \{i'\})=f([n])$ or $f([n]\setminus \{i'\})=f([n])-1$, moreover, $[n]\setminus \{i'\}\in \mathcal{H}_{1}(P)$ if and only if $f([n]\setminus \{i'\})=f([n])-1$. So, by Lemma \ref{coefficients-12}(ii),
\begin{eqnarray*}
[y]X_{P}(y)&=&\sum_{i'\in [n]}f([n]\setminus \{i'\})-(n-1)f([n])\\
&=&f([n])-\sum_{i'\in [n]}\left(f([n])-f([n]\setminus \{i'\})\right)\\
&=&f([n])-|\mathcal{H}_{1}(P)|\\
&=&\binom{f([n])}{1}-\binom{f([n])+1-1-1}{1-1}|\mathcal{H}_{1}(P)|.
\end{eqnarray*}

So, the conclusion is true for $i=0,1$. Note that $r_{2}(P)\leq n$ by Proposition \ref{rP+r'P}(ii). Then the conclusion is true from above discussion for $n=1$. We now assume that $n\geq 2$ and $2\leq i<r_{2}(P)$.
In this case, the following claim holds.

 \textbf{Claim 1.} $|\mathcal{H}_{1}(P)|\leq 1$ if $2<r_{2}(P)$.

\emph{Proof of Claim 1.} Suppose that $|\mathcal{H}_{1}(P)|\geq 2$.  Without loss of generality, we may assume $\{[n]\setminus \{i'\},[n]\setminus \{j'\}\}\subseteq \mathcal{H}_{1}(P)$.  Then $f([n]\setminus \{i'\})=f([n])-1$ and $f([n]\setminus \{j'\})=f([n])-1$. We have that $f([n]\setminus \{i'\})+f([n]\setminus \{j'\})\geq f([n]\setminus \{i',j'\})+f([n])$ by the submodularity of $f$. This implies that $f([n]\setminus \{i',j'\})\leq f([n])-2$, a contradiction with Proposition \ref{rP+r'P}(iii). Hence, Claim 1 is verified.

 By Claim 1,  there exists an index $t\in [n]$ such that
\begin{eqnarray}\label{rank}
f([n]\setminus \{t\})=f([n]),
\end{eqnarray}
that is, $\alpha_{t}=0$. For any $j \leq i$, denote $$\mathcal{H}^{t}_{j}(P)=\{H|~H\in \mathcal{H}_{j}(P)~\text{and}~t\notin H\}.$$

\textbf{Claim 2.} The following conclusion holds.
\begin{enumerate}
\item [(i)] $\{H\cup \{t\}|~H\in \mathcal{H}_{j}(\widehat{P}^{t}_{k})\}=\mathcal{H}_{j}(P)\setminus \mathcal{H}^{t}_{j}(P)$ for any $j$ with $0\leq j \leq i-1$ when $k\in T_{t}\setminus \{0,f(\{t\})\}$ and for any $0\leq j \leq i$ when $k=f(\{t\})$;
\item [(ii)] $\mathcal{H}_{j}(P\setminus t)=(\mathcal{H}_{j}(P)\setminus \mathcal{H}^{t}_{j}(P))\cup \mathcal{H}^{t}_{j+1}(P)$ for any $j$ with $0\leq j \leq i-1$.
\end{enumerate}

\emph{Proof of Claim 2.} Let $f^{t}_{k}$, $f_{P/t}$ and $f_{P\setminus t}$ be the rank functions of $\widehat{P}^{t}_{k}$, $P/t$ and $P\setminus t$, respectively.

(i).
 By Proposition \ref{rank function of Pj}, we have that $f^{t}_{k}([n]\setminus \{t\})=f([n])-k$ for any $k\in T_{t}\setminus \{0\}$. We first prove that if $$f^{t}_{k}(([n]\setminus \{t\})\setminus I)=f([n]\setminus I)-k \ \text{for any}\  t\in [n]\setminus I \ \text{with} \ |I|\leq j,  ~~~~(\star)$$ then $\{H\cup \{t\}|~H\in \mathcal{H}_{j}(\widehat{P}^{t}_{k})\}=\mathcal{H}_{j}(P)\setminus \mathcal{H}^{t}_{j}(P)$.

 For any $[n]\setminus I \in \mathcal{H}_{j}(P)\setminus \mathcal{H}^{t}_{j}(P)$, we have $t\in [n]\setminus I$, $|I|=j$, $f([n]\setminus I)=f([n])-1$  and $f(([n]\setminus I)\cup \{t'\})=f([n])$ for any $t'\in I$. By the condition ($\star$), we have that
\begin{eqnarray*}
f^{t}_{k}(([n]\setminus \{t\})\setminus I)&=&f([n]\setminus I)-k\\
&=&f([n])-1-k\\
&=&f^{t}_{k}([n]\setminus \{t\})-1.
\end{eqnarray*}

Moreover, for any $t'\in I$,
\begin{eqnarray*}
f^{t}_{k}(([n]\setminus \{t\})\setminus I \cup \{t'\})&=&f(([n]\setminus I)\cup \{t'\})-k\\
&=&f([n])-k\\
&=&f^{t}_{k}([n]\setminus \{t\}).
\end{eqnarray*}
Hence, $([n]\setminus \{t\})\setminus I\in \mathcal{H}_{j}(\widehat{P}^{t}_{k})$. This implies $\mathcal{H}_{j}(P)\setminus \mathcal{H}^{t}_{j}(P)\subseteq \{H\cup \{t\}|~H\in \mathcal{H}_{j}(\widehat{P}^{t}_{k})\}$.

Conversely, if $([n]\setminus \{t\})\setminus I \in \mathcal{H}_{j}(\widehat{P}^{t}_{k})$, then $f_{\widehat{P}^{t}_{k}}(([n]\setminus \{t\})\setminus I)=f_{\widehat{P}^{t}_{k}}([n]\setminus \{t\})-1$ and $f_{\widehat{P}^{t}_{k}}(([n]\setminus \{t\})\setminus I\cup \{t'\})=f_{\widehat{P}^{t}_{k}}([n]\setminus \{t\})$ for any $t'\in I$. By the condition ($\star$), $f([n]\setminus I)=f([n])-1$ and $f([n]\setminus I\cup \{t'\})=f([n])$ for any $t'\in I$. This implies $[n]\setminus I\in \mathcal{H}_{j}(P)$. Note that $t\notin I$, i.e., $t\in [n]\setminus I$. Hence, $[n]\setminus I \in \mathcal{H}_{j}(P)\setminus \mathcal{H}^{t}_{j}(P)$.

Hence, $\{H\cup \{t\}|~H\in \mathcal{H}_{j}(\widehat{P}^{t}_{k})\}=\mathcal{H}_{j}(P)\setminus \mathcal{H}^{t}_{j}(P)$.

We next prove that the condition ($\star$) holds for any $j$ with $0\leq j \leq i-1$ when $k\in T_{t}\setminus \{0,f(\{t\})\}$ and for any $0\leq j \leq i$ when $k=f(\{t\})$. By the definition of $P/t$, we have that $f_{P/t}([n]\setminus I)=f([n]\setminus I)-f\{t\}$ for any $I\subseteq [n]\setminus t$. Note that for any $I\subseteq [n]\setminus \{t\}$ with $|I|\leq j$, we have that $f([n]\setminus I\setminus \{t\})\geq f([n])-1$ as $j\leq i-1<r_{2}(P)-1$. Then $f([n]\setminus I)-k\leq f([n]\setminus I)-1\leq f([n])-1\leq f([n]\setminus I\setminus \{t\})$ for any positive integer $k\geq 1$ and for any $I\subseteq [n]\setminus \{t\}$ with $|I|\leq j$.
Hence, Claim 2(i) is verified.

(ii).  By the definition of $P\setminus t$, for any $I\subseteq [n]\setminus t$, we have that
\begin{eqnarray}\label{rank3}
f_{P\setminus t}([n]\setminus I)=f([n]\setminus I).
\end{eqnarray}

We first prove that $(\mathcal{H}_{j}(P)\setminus \mathcal{H}^{t}_{j}(P))\cup \mathcal{H}^{t}_{j+1}(P)\subseteq \mathcal{H}_{j}(P\setminus t)$.

\begin{enumerate}
\item [(A)] If $[n]\setminus I \in \mathcal{H}_{j}(P)\setminus \mathcal{H}^{t}_{j}(P)$, then $t\notin I$, $|I|=j$, $f([n]\setminus I)=f([n])-1$ and $f([n]\setminus I \cup \{t'\})=f([n])$ for any $t'\in I$. Note that $f([n])-1\leq f(([n]\setminus \{t\})\setminus I)\leq f([n]\setminus I)=f([n])-1$ by the monotonicity of $f$. We have that $f(([n]\setminus \{t\})\setminus I)=f([n])-1$. By the submodularity of $f$, we have that $f(([n]\setminus \{t\})\setminus I\cup \{t'\})+f([n]\setminus I)\geq f(([n]\setminus \{t\})\setminus I)+f([n]\setminus I \cup \{t'\})$. Then $f(([n]\setminus \{t\})\setminus I\cup \{t'\})=f([n])$. By Equations (\ref{rank}) and (\ref{rank3}), we have that $f_{P\setminus t}(([n]\setminus \{t\})\setminus I)=f_{P\setminus t}([n]\setminus \{t\})-1$ and $f_{P\setminus t}(([n]\setminus \{t\})\setminus I\cup \{t'\})=f_{P\setminus t}([n]\setminus \{t\})$ for any $t'\in I$. This implies $([n]\setminus \{t\})\setminus I \in \mathcal{H}_{j}(P\setminus t)$.

\item [(B)] If $[n]\setminus I \in \mathcal{H}^{t}_{j+1}(P)$, then $t\in I$, $|I|=j+1$, $f([n]\setminus I)=f([n])-1$ and $f([n]\setminus I \cup \{t'\})=f([n])$ for any $t'\in I$. By Equations (\ref{rank}) and (\ref{rank3}), we have that $f_{P\setminus t}(([n]\setminus \{t\})\setminus (I\setminus \{t\}))=f_{P\setminus t}([n]\setminus \{t\})-1$ and
 $f_{P\setminus t}(([n]\setminus \{t\})\setminus (I\setminus \{t\})\cup \{t'\})=f_{P\setminus t}([n]\setminus \{t\})$ for any $t'\in I\setminus \{t\}$. Hence, $([n]\setminus \{t\})\setminus (I\setminus \{t\}) \in \mathcal{H}_{j}(P\setminus t)$.
\end{enumerate}
In conclusion, $(\mathcal{H}_{j}(P)\setminus \mathcal{H}^{t}_{j}(P))\cup \mathcal{H}^{t}_{j+1}(P)\subseteq \mathcal{H}_{j}(P\setminus t)$.

Conversely, if $([n]\setminus \{t\})\setminus I \in \mathcal{H}_{j}(P\setminus t)$, then $|I|=j$, $f_{P\setminus t}(([n]\setminus \{t\})\setminus I)=f_{P\setminus t}([n]\setminus \{t\})-1$ and $f_{P\setminus t}(([n]\setminus \{t\})\setminus I\cup \{t'\})=f_{P\setminus t}([n]\setminus \{t\})$ for any $t'\in I$. By Equations (\ref{rank}) and (\ref{rank3}), we have that $f(([n]\setminus \{t\})\setminus I)=f([n])-1$ and $f(([n]\setminus \{t\})\setminus I\cup \{t'\})=f([n])$ for any $t'\in I$.
\begin{enumerate}
\item [(C)] If $f([n]\setminus I)=f([n])-1$, then for any $t'\in I$, $f([n])\geq f(([n]\setminus I)\cup \{t'\}) \geq f(([n]\setminus \{t\})\setminus I\cup \{t'\})=f([n])$ by the monotonicity of $f$. This implies $f(([n]\setminus I)\cup \{t'\})=f([n])$. Hence, $[n]\setminus I\in \mathcal{H}_{j}(P)\setminus \mathcal{H}^{t}_{j}(P)$.

\item [(D)] If $f([n]\setminus I)=f([n])$, then $f([n]\setminus (I\cup \{t\}))=f([n])-1$ and $f(([n]\setminus \{t\})\setminus I\cup \{t'\})=f([n])$ for any $t'\in I\cup \{t\}$. This implies $[n]\setminus (I\cup \{t\})\in \mathcal{H}^{t}_{j+1}(P)$.
\end{enumerate}

Hence, $\mathcal{H}_{j-1}(P\setminus t)\subseteq(\mathcal{H}_{j}(P)\setminus \mathcal{H}^{t}_{j}(P))\cup \mathcal{H}^{t}_{j+1}(P)$.
We complete the proof of Claim 2(ii).

By the induction hypothesis, Claim 2 and Lemma \ref{deletion-contraction-IX}, we know that
\begin{eqnarray*}
&&[y^{i}]X_{P}(y)\\
&=&[y^{i}]X_{P/t}(y)+\sum_{k\in T_{t}\setminus \{f(\{t\})\}}[y^{i-1}]X_{\widehat{P}^{t}_{k}}(y)\\
&=&\binom{f([n])-f(\{t\})+i-1}{i}-\sum_{j=0}^{i}\binom{f([n])-f(\{t\})+i-1-j}{i-j}|\mathcal{H}_{j}(P/t)|\\
&+&\sum_{k\in T_{t}\setminus \{f(\{t\})\}}\left(\binom{f([n])-k+i-2}{i-1}-\sum_{j=0}^{i-1}\binom{f([n])-k+i-2-j}{i-1-j}|\mathcal{H}_{j}(\widehat{P}^{t}_{k})|\right)\\
&=&\left(\binom{f([n])-f(\{t\})+i-1}{i}+\sum_{k\in T_{t}\setminus \{f(\{t\})\}}\binom{f([n])-k+i-2}{i-1}\right)\\
&-&\sum_{j=0}^{i}\binom{f([n])-f(\{t\})+i-1-j}{i-j}(|\mathcal{H}_{j}(P)|-|\mathcal{H}^{t}_{j}(P)|)\\
&-&\sum_{k\in T_{t}\setminus \{f(\{t\}),0\}}\left(\sum_{j=0}^{i-1}\binom{f([n])-k+i-2-j}{i-1-j}(|\mathcal{H}_{j}(P)|-|\mathcal{H}^{t}_{j}(P)|)\right)\\
&-&\sum_{j=0}^{i-1}\binom{f([n])+i-2-j}{i-1-j}(|\mathcal{H}_{j}(P)|-|\mathcal{H}^{t}_{j}(P)|+|\mathcal{H}^{t}_{j+1}(P)|)\\
&=&\binom{f([n])+i-1}{i}-\sum_{j=0}^{i}\binom{f([n])-f(\{t\})+i-1-j}{i-j}(|\mathcal{H}_{j}(P)|-|\mathcal{H}^{t}_{j}(P)|)\\
&-&\sum_{j=0}^{i-1}\sum_{k\in T_{t}\setminus \{f(\{t\}),0\}}\binom{f([n])-k+i-2-j}{i-1-j}(|\mathcal{H}_{j}(P)|-|\mathcal{H}^{t}_{j}(P)|)\\
&-&\sum_{j=0}^{i-1}\binom{f([n])+i-2-j}{i-1-j}(|\mathcal{H}_{j}(P)|-|\mathcal{H}^{t}_{j}(P)|+|\mathcal{H}^{t}_{j+1}(P)|)\\
&=&\binom{f([n])+i-1}{i}-\sum_{j=0}^{i-1}\binom{f([n])+i-1-j}{i-j}|\mathcal{H}_{j}(P)|-|\mathcal{H}_{i}(P)|+|\mathcal{H}^{t}_{i}(P)|\\
&+&\sum_{j=0}^{i-1}\binom{f([n])+i-1-j}{i-j}|\mathcal{H}^{t}_{j}(P)|-\sum_{j=0}^{i-1}\binom{f([n])+i-2-j}{i-1-j}|\mathcal{H}^{t}_{j+1}(P)|\\
&=&\binom{f([n])+i-1}{i}-\sum_{j=0}^{i-1}\binom{f([n])+i-1-j}{i-j}|\mathcal{H}_{j}(P)|-|\mathcal{H}_{i}(P)|+|\mathcal{H}^{t}_{i}(P)|\\
&+&\sum_{j=0}^{i-1}\binom{f([n])+i-1-j}{i-j}|\mathcal{H}^{t}_{j}(P)|-\sum_{j=1}^{i}\binom{f([n])+i-1-j}{i-j}|\mathcal{H}^{t}_{j}(P)|\\
&=&\binom{f([n])+i-1}{i}-\sum_{j=0}^{i-1}\binom{f([n])+i-1-j}{i-j}|\mathcal{H}_{j}(P)|-|\mathcal{H}_{i}(P)|+|\mathcal{H}^{t}_{i}(P)|\\
&+&\binom{f([n])+i-1}{i}|\mathcal{H}^{t}_{0}(P)|-|\mathcal{H}^{t}_{i}(P)|\\
&=&\binom{f([n])+i-1}{i}-\sum_{j=0}^{i}\binom{f([n])+i-1-j}{i-j}|\mathcal{H}_{j}(P)|.
\end{eqnarray*}
So, the conclusion (i) holds.

We next prove the conclusion (ii). Let $P^{\ast}$ be the polymatroid given by  $f^{\ast}$ (see Equation (\ref{dual-polymatroid})). Note that $f^{\ast}([n])=\sum _{i'\in [n]}f(\{i'\})-f([n])$, and $[x^{i}]I_{P}(x)=[y^{i}]X_{P^{\ast}}(y)$ by Equation (\ref{duality-relation}).
 Hence, the conclusion (ii) holds from the conclusion (i) and Proposition \ref{P+P*}.
\end{proof}

\begin{remark} Our main theorem Theorem \ref{connected-exterior2} includes both Theorem \ref{connected} and Theorem \ref{connected-exterior}.
\begin{itemize}
\item[(a)]If $P=P(M)$ is the corresponding polymatroid of a matroid $M$, then $r_{2}(P(M))=f_{2}(M)$, $r'_{2}(P(M))=f'_{2}(M)$, $\mathcal{H}_{j}(P(M))=\mathcal{H}_{j}(M)$, $\mathcal{C}_{j}(P(M))=\mathcal{C}_{j}(M)$. Hence, Theorem \ref{connected-exterior2} is a generalization of Theorem \ref{connected} from matroids to polymatroids.
\item[(b)]It is obvious that for any integer $k$, we have that $|\mathcal{H}_{j}(P)|=0$ for any $j\leq k$ if and only if $f([n]\setminus J)=f([n])$ for all $J\subseteq [n]$ with $|J|=k$. So, Theorem \ref{connected-exterior} is a special case of Theorem \ref{connected-exterior2}.
    \end{itemize}
    \end{remark}

\begin{example}\label{EX}
Let $P\subseteq \mathbb{Z}^{5}$ be a polymatroid. The rank function $f:2^{\{1,2,3,4,5\}}\rightarrow \mathbb{Z}$ of $P$ is given by $f(\emptyset)=0$, $f(\{1\})=f(\{3\})=1$, $f(\{2\})=f(\{4\})=f(\{5\})=f(\{1,2\})=f(\{1,3\})=f(\{2,3\})=f(\{4,5\})=f(\{1,2,3\})=2$ and $f(I)=3$ for any other $I\subseteq \{1,2,3,4,5\}$. It is easy to see that $r_{2}(P)=4$, $\mathcal{H}_{0}(P)=\mathcal{H}_{1}(P)=\emptyset$, $\mathcal{H}_{2}(P)=\{\{1,2,3\}\}$, $\mathcal{H}_{3}(P)=\{\{4,5\}\}$.
By \cite[Example 3.4]{Guan6}, we know that
$$X_{P}(y)=1+3y+5y^{2}+6y^{3}+3y^{4}+y^{5}.$$
Hence, $[y]X_{P}(y)=3=\binom{3}{1}$, $[y^{2}]X_{P}(y)=5=\binom{4}{2}-|\mathcal{H}_{2}(P)|$ and $[y^{3}]X_{P}(y)=6=\binom{5}{3}-\binom{3}{1}|\mathcal{H}_{2}(P)|-|\mathcal{H}_{3}(P)|$.
\end{example}

Let $G$ be a simple graph. For a vertex $v$, let $deg_{G}(v)$ denote the number of edges incident with $v$ in the graph $G$. Then Theorem \ref{connected-exterior2} implies the following result for hypergraphical polymatroids.
\begin{corollary}
Let $\mathcal{H}=(V,E)$ be a connected hypergraph and let $\mathrm{Bip} \mathcal{H}$ be its associated bipartite graph. Denote $r_{2}(\mathcal{H})=\min\{|H'|:~c(\mathrm{Bip} \mathcal{H}|_{E\setminus H'})\geq 3 ~\text{for}~H'\subseteq E\}$, $\mathcal{H}_{j}(\mathcal{H})=\{H'\subseteq E:~c(\mathrm{Bip} \mathcal{H}|_{E\setminus H'})=2,  c(\mathrm{Bip} \mathcal{H}|_{(E\setminus H')\cup \{e\}})=1~\text{for any} ~ e\in E\setminus H'~ \text{and}~|H|=j\}$, $g=\sum_{e\in E}deg_{\mathrm{Bip} \mathcal{H}}(e)-|E|-|V|$, $r'_{2}(\mathcal{H})=\min\{|H'|:$ there are at least 2 cycles in $\mathrm{Bip} \mathcal{H}|_{H'}\}$ and $\mathcal{C}_{j}(\mathcal{H})=\{H'\subseteq E:~\mathrm{Bip} \mathcal{H}|_{H'}$ has exactly one cycle and its length is~$2j=2|H'|$$\}$. Then

\begin{itemize}
  \item [(i)] for all $0\leq i<r_{2}(\mathcal{H})$, $$[y^{i}]X_{\mathcal{H}}(y)=\binom{|V|-2+i}{i}-\sum_{j=0}^{i}\binom{|V|+i-2-j}{i-j}|\mathcal{H}_{j}(\mathcal{H})|;$$
  \item [(ii)] for all $0\leq i <r'_{2}(\mathcal{H})$, $$[x^{i}]I_{\mathcal{H}}(x)=\binom{g+i-2}{i}-\sum_{j=0}^{i}\binom{g+i-1-j}{i-j}|\mathcal{C}_{j}(\mathcal{H})|.$$
\end{itemize}
\end{corollary}

\begin{proof}
Let $P_{\mathcal{H}}$ be the polymatroid induced by $\mathcal{H}$ and let $f$ be the rank function of  $P_{\mathcal{H}}$. It is clear that $f(E)=|V|-1$, $r_{2}(\mathcal{H})=r_{2}(P_{\mathcal{H}})$, $r'_{2}(\mathcal{H})=r'_{2}(P_{\mathcal{H}})$, $\mathcal{H}_{j}(\mathcal{H})=\mathcal{H}_{j}(P_{\mathcal{H}})$, $\mathcal{C}_{j}(\mathcal{H})=\mathcal{C}_{j}(P_{\mathcal{H}})$ and
\begin{eqnarray*}
\sum_{e\in E}f(e)&=&\sum_{e\in E}(deg_{\mathrm{Bip} \mathcal{H}}(e)-1)\\
&=& \sum_{e\in E}deg_{\mathrm{Bip} \mathcal{H}}(e)-|E|.
\end{eqnarray*}
 Hence, the conclusion follows from Theorem \ref{connected-exterior2}.
\end{proof}

In Theorem \ref{connected-exterior2}, it is obvious that for any integer $k$, $|\mathcal{C}_{j}(P)|=0$ for any $j\leq k$ if and only if $\sum _{i'\in J}f(\{i'\})=f(J)$ for all $J\subseteq [n]$ with $|J|=k$. Hence, the following conclusion is true.
\begin{corollary} \label{connected-interior}
Let $P\subseteq \mathbb{Z}^{n}_{\geq 0}$ be a polymatroid and let $f$ be its rank function. Then for all $k\leq n-1$, the coefficient $$[x^{i}]I_{P}(x)=\binom{\sum _{i'\in [n]}f(\{i'\})-f([n])+i-1}{i}$$ for all $i\leq k$ if and only if $\sum _{i'\in J}f(\{i'\})=f(J)$ for all $J\subseteq [n]$ with $|J|=k$.
\end{corollary}

We end this section with an immediate consequence of Corollary \ref{connected-interior} for hypergraphical polymatroids.
\begin{corollary}
Let $\mathcal{H}=(V,E)$ be a connected hypergraph and let $\mathrm{Bip} \mathcal{H}$ be its associated bipartite graph. Then for all $k\leq |E|$, we have $[x^{i}]I_{\mathcal{H}}=\binom{\sum_{e\in E}deg_{\mathrm{Bip} \mathcal{H}}(e)-|E|-|V|+i-2}{i}$ for all $i\leq k$ if and only if the length of a shortest cycle of $\mathrm{Bip} \mathcal{H}$ is at least $2k+2$.
\end{corollary}

\section{Unimodality}
\noindent

In this section, as an application we study the unimodality of coefficients of exterior and interior polynomials obtained in Theorem \ref{connected-exterior2}. First, we recall the definition of the unimodality of a sequence.
\begin{definition}
For a  sequence $a_{0}$,$a_{1}$,$\ldots$,$a_{n}$ of real numbers, we call the sequence \emph{unimodal} if there exists an integer $k$ such that $a_{i}\leq a_{i+1}$ when $i<k$, and $a_{i}\geq a_{i+1}$ when $i>k$.
\end{definition}
\begin{theorem}\label{unimodality}
Let $P\subseteq \mathbb{Z}^{n}_{\geq 0}$ be a polymatroid. Then
 \begin{itemize}
  \item [(i)] $[y^{0}]X_{P}(y)$,$[y^{1}]X_{P}(y)$,$\ldots$,$[y^{r_{2}(P)-1}]X_{P}(y)$ is unimodal;
  \item [(ii)] $[x^{0}]I_{P}(x)$,$[x^{1}]I_{P}(x)$,$\ldots$,$[x^{r'_{2}(P)-1}]I_{P}(x)$ is unimodal.
\end{itemize}
\end{theorem}

\begin{proof}
By Equation (\ref{duality-relation}), it suffices to prove (i). Let $f$ be the rank function of $P$.

If $f([n])=0$, then $f(I)=0$ for any $I\subseteq [n]$ by the monotonicity of $f$. We have that $P=\{(0,0,\ldots,0)\}$. This implies $X_{P}(y)=1$. The conclusion holds.

Suppose $f([n])=1$. We may assume that $f(\{i\})\geq 1$ for any $i\in [n]$. (If there is an $i'$ such that $f(\{i'\})=0$, then $\alpha_{i'}=\beta_{i'}=0$. By Lemma \ref{deletion-contraction-IX}, $X_{P}=X_{P\setminus i'}$.)  Then $f(I)=1$ for any $\emptyset \neq I\subseteq [n]$ by the monotonicity of $f$. We have that $P=\{\textbf{e}_{1},\textbf{e}_{2},\ldots,\textbf{e}_{n}\}$. This implies $X_{P}(y)=1+y+y^{2}+\ldots+y^{n-1}$. In this case, the conclusion also holds.

We now assume that $f([n])\geq 2$. It is enough to prove the following claim.

\textbf{Claim.} For any $1\leq i<r_{2}(P)-1$, if $[y^{i}]X_{P}(y)-[y^{i-1}]X_{P}(y)\leq0$, then $[y^{i+1}]X_{P}(y)-[y^{i}]X_{P}(y)\leq 0$.

\emph{Proof of Claim 1.} From Theorem \ref{connected-exterior2}, we have that
$$[y^{i-1}]X_{P}(y)=\binom{f([n])+i-2}{i-1}-\sum_{j=0}^{i-1}\binom{f([n])+i-2-j}{i-1-j}|\mathcal{H}_{j}(P)|$$ and
$$[y^{i+1}]X_{P}(y)=\binom{f([n])+i}{i+1}-\sum_{j=0}^{i+1}\binom{f([n])+i-j}{i+1-j}|\mathcal{H}_{j}(P)|.$$

Note that $\binom{f([n])-2}{0}=\binom{f([n])-1}{0}=1$ if $f([n])\geq 2$. Then
\begin{eqnarray*}
&&[y^{i}]X_{P}(y)-[y^{i-1}]X_{P}(y)\\
&=&\left(\binom{f([n])+i-1}{i}-\sum_{j=0}^{i}\binom{f([n])+i-1-j}{i-j}|\mathcal{H}_{j}(P)|\right)\\
&-&\left(\binom{f([n])+i-2}{i-1}-\sum_{j=0}^{i-1}\binom{f([n])+i-2-j}{i-1-j}|\mathcal{H}_{j}(P)|\right)\\
&=&\binom{f([n])+i-2}{i}-\sum_{j=0}^{i}\binom{f([n])+i-2-j}{i-j}|\mathcal{H}_{j}(P)|.
\end{eqnarray*}
Note that for any integers $a,b,j$, we have that $\frac{a-j}{b-j}\geq\frac{a}{b}$ if $a\geq b>j$, moreover, $\binom{a+1}{b+1}=\frac{a+1}{b+1}\binom{a}{b}$. Then
\begin{eqnarray*}
&&[y^{i+1}]X_{P}(y)-[y^{i}]X_{P}(y)\\
&=&\binom{f([n])+i-1}{i+1}-\sum_{j=0}^{i+1}\binom{f([n])+i-1-j}{i+1-j}|\mathcal{H}_{j}(P)|\\
&\leq&\binom{f([n])+i-1}{i+1}-\sum_{j=0}^{i}\binom{f([n])+i-1-j}{i+1-j}|\mathcal{H}_{j}(P)|\\
&=&\frac{f([n])+i-1}{i+1}\binom{f([n])+i-2}{i}\\
&-&\sum_{j=0}^{i}\frac{f([n])+i-1-j}{i+1-j}\binom{f([n])+i-2-j}{i-j}|\mathcal{H}_{j}(P)|\\
&\leq& \frac{f([n])+i-1}{i+1}\left(\binom{f([n])+i-2}{i}-\sum_{j=0}^{i}\binom{f([n])+i-2-j}{i-j}|\mathcal{H}_{j}(P)|\right)\\
&=& \frac{f([n])+i-1}{i+1}([y^{i}]X_{P}(y)-[y^{i-1}]X_{P}(y))\\
&\leq& 0.
\end{eqnarray*}
\end{proof}

\begin{remark}
Recall that Theorem \ref{connected-exterior} is a special case of Theorem \ref{connected-exterior2}. In 2018, Adiprasito, Huh and Katz \cite{Adiprasito} indirectly proved that the sequences of all coefficients of $T_{M}(x,1)$ and $T_{M}(1,y)$ are both unimodal by applying Hodge theory. Hence, Theorem \ref{unimodality}  provides a simpler and direct proof for the unimodality of partial coefficients of $T_{M}(x,1)$ and $T_{M}(1,y)$.
\end{remark}

\section*{Acknowledgements}
\noindent

This work was supported by National Natural Science Foundation of China (Nos.~12171402, 12371356 and 12401462), the Natural Science Foundation of Shanxi Province (No.~202403021222034), Shanxi Key Laboratory of Digital Design and Manufacturing and Scientific Research Start-Up Foundation of Jimei University (No.~ZQ2024116).

\section*{References}
\bibliographystyle{model1b-num-names}
\bibliography{<your-bib-database>}

\end{document}